\documentclass[letterpaper, 10 pt, conference]{ieeeconf}  

\IEEEoverridecommandlockouts                              
\usepackage{amsmath,amssymb,theorem,subfigure}

\usepackage{tikz}
\usepackage{xypic}
\usepackage{cases}
\usepackage{epstopdf,graphicx}

\newtheorem{theorem}{Theorem}[section]

\newtheorem{lemma}[theorem]{Lemma}

\newtheorem{remark}[theorem]{Remark}

\newcommand{\Sym}{\mathrm{Sym}}
\newcommand{\Skew}{\mathrm{Skew}}

\newcommand{\SO}{\operatorname{SO}(3)}
\newcommand{\MI}{{\mathbb I}}
\newcommand{\dimu}{{k}}
\newcommand{\coV}{{k_{\rm e}}}

\title{\LARGE \bf
Controller Design for Systems on Manifolds in Euclidean Space
}

\author{Dong Eui Chang
\thanks{D.E. Chang is with the school of electrical engineering, KAIST, Korea and the department of applied mathematics, University of Waterloo, Canada.
        {\tt\small dechang@kaist.ac.kr}. This research has been in part supported by KAIST  under grant G04170001 and  by the ICT R\&D program of MSIP/IITP [2016-0-00563, Research on Adaptive Machine Learning Technology Development for Intelligent Autonomous Digital Companion]. This paper is presented at 56th IEEE Conference on Decision and Control, Melbourne, Australia, December 2017.}
}

\begin{document}

\maketitle
\thispagestyle{empty}
\pagestyle{empty}

\begin{abstract}
Given a control system on a manifold that is embedded in Euclidean space, it is sometimes convenient to use a single global coordinate system in the ambient Euclidean space for controller design  rather than to use multiple local charts on the manifold or  coordinate-free tools from differential geometry. In this paper, we develop a theory about this and apply it to the fully actuated rigid body system for stabilization and tracking. A noteworthy  point in this theory is that   we legitimately modify the system dynamics outside its state-space manifold before controller design so as to add  attractiveness  to the manifold in the resulting  dynamics.

\end{abstract}

\section{Introduction}
Quite a few control systems are defined on manifolds that are not homeomorphic to Euclidean space, where we  use the phrase `Euclidean space' to mean some $\mathbb R^n$ space, not imposing any metric on it. The geometric, or coordinate-free, approach has been developed to deal with such systems without being dependent on the choice of coordinates. In many  cases, however, a state-space manifold appears as an embedded manifold in Euclidean space and the control system naturally extends from the manifold to its ambient Euclidean space: one example is the free rigid body system on $\SO \times \mathbb R^3$ which  naturally extends to  $\mathbb R^{3\times 3} \times \mathbb R^3$.  In such a case, it might be advantageous to use one single global Cartesian coordinate system in the ambient Euclidean space in designing controllers for the system on the manifold, thus eliminating the necessity of using multiple local charts  or rather complex tools from differential geometry. 
When the state-space manifold, say $M$,  is a leaf of a foliation of invariant manifolds of the extended, or ambient, system, we can legitimately modify the ambient system dynamics outside $M$ to add attractiveness to $M$ while preserving the dynamics on $M$;  design controllers for the modified ambient system in the ambient Euclidean space; and then apply  the resultant controllers to the original system on $M$. In this paper, we showcase this program in combination with the linearization technique; the usual Jacobian linearization is carried out on  the ambient system  to come up with  stabilizing or tracking controllers for the original system on the manifold. The free rigid body system is used here to illustrate every step of this program. We note that  the program of  using ambient Euclidean space was successfully employed in providing a simple proof of  the Pontryagin Maximum Principle on manifolds \cite{Ch11} and creating feedback integrators for structure-preserving numerical integration \cite{ChJiPe16}. 

\section{Main Results}
\subsection{Notation and Some Mathematical Facts}
The usual Euclidean inner product is exclusively used for vectors and matrices in this paper, i.e.
\[
\langle A, B \rangle = \sum_{i,j}A_{ij}B_{ij} = \operatorname{tr}(A^TB)
\]
for any two matrices  of equal size.  The norm induced from this inner product, which is called the Frobenius or Euclidean norm, is exclusively used for vectors and matrices. Let $\Sym$ and $\Skew$ be the symmetriztaion operator and the skew-symmetrization operator, respectively, on square matrices, which are defined by
\[
\Sym (A) = \frac{1}{2}(A+A^T), \quad \Skew (A) = \frac{1}{2}(A-A^T)
\]
for any square matrix $A$.
Then,
\[
A  = \Sym (A) + \Skew(A), \quad \langle \Sym (A), \Skew (A) \rangle = 0.
\]
 Namely, 
\[
\mathbb R^{n\times n} = \Sym (\mathbb R^{n\times n} ) \oplus \Skew (\mathbb R^{n\times n} )
\]
with respect to the Euclidean inner product. Let $[\, , ]$ denote the usual matrix commutator that is defined by $[A,B] = AB-BA$ for any pair of square matrices of equal size. 
It is easy to show that
\begin{align*}
[\Sym (\mathbb R^{n\times n} ), \Sym (\mathbb R^{n\times n} )] &\subset \Skew (\mathbb R^{n\times n} ),\\
[\Sym (\mathbb R^{n\times n} ), \Skew (\mathbb R^{n\times n} )] &\subset \Sym (\mathbb R^{n\times n} ),\\
[\Skew (\mathbb R^{n\times n} ), \Skew (\mathbb R^{n\times n} )] &\subset \Skew (\mathbb R^{n\times n} ).
\end{align*}
In other words,  $[A,B] = -[A,B]^T$  for any $A = A^T \in \mathbb R^{n\times n}$ and $B = B^T \in \mathbb R^{n\times n}$; $[A,C] = [A,C]^T$ for any  $A = A^T \in \mathbb R^{n\times n}$ and $C = -C^T \in \mathbb R^{n\times n}$; and $[B,C] = -[B,C]^T$   for any  $B = -B^T \in \mathbb R^{n\times n}$ and  $C = -C^T \in \mathbb R^{n\times n}$. 

Let $\SO$ denote the set of all $3\times 3$ rotation matrices, which is defined as $\SO = \{ R\in \mathbb R^{3\times 3} \mid R^T R - I= 0, \det R>0\}$.  Let $\mathfrak{so}(3)$ denote the set of all $3\times 3$ skew symmetric matrices, which is defined as $\mathfrak{so}(3) = \{ A \in \mathbb R^{3\times 3} \mid A^T+ A = 0 \}$. The hat map $\wedge : \mathbb R^3 \rightarrow \mathfrak{so}(3)$ is defined by
\[
\hat \Omega = \begin{bmatrix}
0 & -\Omega_3 &\Omega_2 \\
\Omega_3 & 0 & -\Omega_1\\
-\Omega_2 & \Omega_1 & 0
\end{bmatrix}
\]
for $\Omega = (\Omega_1, \Omega_2,\Omega_3) \in \mathbb R^3$. The inverse map of the hat map is called the vee map and denoted by $\vee$ such that $(\hat \Omega)^\vee = \Omega$ for all $\Omega \in \mathbb R^3$ and $(A^\vee)^\wedge = A$ for all $A\in \mathfrak{so}(3)$.  

\begin{lemma}\label{lemma:prelim}
1. $\langle R A, R B\rangle = \langle AR, BR\rangle  = \langle A, B\rangle$   for any $R \in \SO$ and $A,B \in \mathbb R^{3\times 3}$.




2. $\max_{R_1, R_2 \in \SO}\| R_1 - R_2\| = 2\sqrt 2$. 

3. $\langle \hat u, \hat v \rangle = 2 \langle u, v\rangle$ for any $u, v \in \mathbb R^3$.

4. $[\hat u, \hat v]= (u \times v)^\wedge$  and $\hat u v = u \times v$ for any $u, v \in \mathbb R^3$.
\end{lemma}
 Every function and manifold is assumed to be smooth in this paper unless stated otherwise. Stability, stabilization and tracking are all understood to be local unless globality is stated explicitly. 

\subsection{The Setup}
Consider a control system $\Sigma$ on $\mathbb R^n$
\[
\Sigma: \quad \dot x = X(x,u), \quad x \in \mathbb R^n, u \in \mathbb R^\dimu.
\]
Assume that there is an $m$-dimensional regular submanifold $M$ of $\mathbb R^n$ that is invariant under the flow of the system $\Sigma$. 
 By  the invariance of $M$, we can restrict the system $\Sigma$ to $M$ and denote the restricted system by $\Sigma | M$ as follows:
\begin{equation}\label{Sigma:M}
\Sigma | M: \quad \dot x = X(x,u), \quad x \in  M, u \in \mathbb R^\dimu.
\end{equation}
For convenience, we will call the system $\Sigma$ an ambient system of $\Sigma | M$. Any control system on $\mathbb R^n$ whose restriction to $M$ coincides with $\Sigma | M$ shall be also called an ambient system of $\Sigma |M$.

A control system is often defined on  a manifold   and  given in a form embedded in Euclidean space as above. Hence, it will be convenient to use the ambient control system $\Sigma$  and the Cartesian coordinates on the ambient space $\mathbb R^n$   in order to design controllers for the system $\Sigma | M$ on the manifold $M$, which may  free us from using multiple local charts or  difficult tools from differential geometry.  For example, the bracket operations on control vector fields related to the control system $\Sigma | M$ can be carried out in Cartesian coordinates in $\mathbb R^n$ since those vector fields can be regarded as ones in $\mathbb R^n$ and the bracket operation is closed in the tangent bundle $TM$ of the manifold $M$.  Optimal control problems on $M$ can be also solved in the ambient space $\mathbb R^n$  about which we  refer the reader to \cite{Ch11}.

Since we are interested in the system $\Sigma | M$ on $M$, it is acceptable to modify its ambient system  $\Sigma$ outside $\Sigma | M$ while preserving the dynamics on $M$. Suppose that there is a non-negative function $\tilde V$ on $\mathbb R^n$ such that $M = \tilde V^{-1}(0)$. A natural candidate for $\tilde V$ would be $\tilde V(x) = \frac{1}{2}f(x)^TSf(x)$, where $f: \mathbb R^n \rightarrow \mathbb R^q$ with $q \geq n -\dim M$ is a function such that $M = f^{-1}(0)$, and $S$ is a $q \times q$ positive definite symmetric matrix. Since the function $\tilde V$ attains its minimum value $0$ at every point in  $M$, 
\begin{equation}\label{nabla:V:0}
\nabla \tilde V(x) = 0, \quad \forall x \in M.
\end{equation}
Subtract $\nabla \tilde V$ from the control vector field of $\Sigma$ to obtain the following new ambient control system
\begin{align}\label{Sigma:tilde}
\tilde \Sigma: \quad \dot x& = \tilde X(x,u), \quad x \in \mathbb R^n, u \in \mathbb R^\dimu,
\end{align}
where
\[
\tilde X(x,u) =X(x,u) -\nabla \tilde V(x).
\]
By \eqref{nabla:V:0}, the two systems $\Sigma$ and $\tilde \Sigma$ coincide on $M$, i.e. $\Sigma | M = \tilde \Sigma | M$, so we will mainly use $\tilde \Sigma$ in place of $\Sigma$ as the ambient system of $\Sigma | M$.  The role of the term $-\nabla \tilde V$ is to help making $M$  attractive in the dynamics of $\tilde \Sigma$. We refer the reader   to \cite{ChJiPe16} for conditions for attractiveness of $M$  in the dynamics of $\tilde \Sigma$.

As an example throughout the paper, we use the following  free rigid body system with full actuation given by
\begin{subequations}\label{rigid:original}
\begin{align}
\dot R &= R\hat \Omega, \\
\dot \Omega &= \MI^{-1} ( \MI \Omega \times \Omega) + \MI^{-1} \tau,
\end{align}
\end{subequations}
where $(R,\Omega) \in \SO \times  \mathbb R^3 \subset \mathbb R^{3\times 3} \times \mathbb R^3$ is the state vector consisting of a rotation matrix $R$ and a body angular velocity $\Omega$; $\tau \in \mathbb R^3$ is the control torque; and $\mathbb I$ is the moment of inertial matrix of the rigid body. From here on, we regard the system \eqref{rigid:original} as a system defined on  $\mathbb R^{3\times 3} \times \mathbb R^3$, treating $R$ as a $3\times 3$ matrix.  It is then easy to verify that $\SO \times \mathbb R^3$ is an invariant set of \eqref{rigid:original}, i.e. every flow starting in $M$ stays in $M$ for all $t\in \mathbb R$.
Assume that the full state of the system is available, which allows us to apply the following controller 
\[
\tau = \MI (u - \MI^{-1}(\MI \Omega \times \Omega))
\]
to transform the above system to
\begin{subequations}\label{rigid:eq}
\begin{align}
\dot R &= R\hat \Omega, \label{R:eq}\\
\dot \Omega &= u, \label{Omega:eq}
\end{align}
\end{subequations}
where $u$ is the new control vector. Note that $\SO \times \mathbb R^3$ is an invariant set of \eqref{rigid:eq}.  Let $W = \{ R \in \mathbb R^{3\times 3} \mid \det R >0\}$ and define a function $\tilde V$ on $W \times \mathbb R^3\subset \mathbb R^{3\times 3} \times\mathbb R^3$ by 
\[
\tilde V(R,\Omega) = \frac{\coV }{4}\|R^TR - I\|^2,
\]
where $\coV >0$. It is easy to verify that $\tilde V^{-1}(0) = \SO\times \mathbb R^3$ and 
\[
\nabla_R \tilde V = -\coV  R(R^TR - I), \quad \nabla_\Omega \tilde V = 0.
\]
With this function $\tilde V$,  the modified rigid body system $\tilde \Sigma$ corresponding to \eqref{Sigma:tilde} is computed as
\begin{subequations}\label{rigid:tilde:eq}
\begin{align}
\dot R &= R\hat \Omega - \coV R(R^TR - I), \label{R:s:eq}\\
\dot \Omega &= u, \label{Omega:s:eq}
\end{align}
\end{subequations}
where $(R,\Omega) \in \mathbb R^{3\times 3} \times \mathbb R^3$.

\subsection{Stabilization via Linearization in Ambient Euclidean Space}
Consider the system $\tilde \Sigma$ given in \eqref{Sigma:tilde} and its restriction $\Sigma | M$ to $M$ given in \eqref{Sigma:M}, where  $\Sigma | M =\tilde  \Sigma | M$ is understood. Let $(x_0,u_0) \in M \times \mathbb R^\dimu$ be an equilibrium point of $\Sigma |M$, i.e, $X(x_0,u_0) = 0$. 
Suppose that we want to approximate the dynamics of $\Sigma | M$ by (Jacobian) linearization at  $(x_0, u_0)$.  One would normally choose a local coordinate chart on $M$ containing the point $x_0$,  express the dynamics of $\Sigma |M$ on this chart and then linearize it on the chart, where rewriting the  equations of motion on the local chart could be regarded as an extra process. In addition to localness of linearization, the use of a local chart may cause  extra localness. To remedy it, we here propose to carry out Jacobian linearization of the ambient system $\tilde \Sigma $ at the equilibrium point $(x_0, u_0)$ in the ambient space $\mathbb R^n \times \mathbb R^\dimu$, instead.  The linearization $\tilde \Sigma^\ell_0$ of $\tilde \Sigma $ is given by
\[
\tilde \Sigma^\ell_0: \quad \dot x = \frac{\partial \tilde X}{\partial x}(x_0,u_0) (x-x_0) +  \frac{\partial \tilde X}{\partial u}(x_0,u_0) (u-u_0),
\]
where $(x,u) \in \mathbb R^n \times \mathbb R^\dimu$. Notice that the difference $x-x_0$ would not make  sense on $M$, but it  does make  perfect sense in the ambient Euclidean space $\mathbb R^n$.

The following lemma is trivial but useful:
\begin{lemma}\label{lemma:exp:stab}
If a feedback controller $u: \mathbb R^n \rightarrow \mathbb R^\dimu$ stabilizes, in any sense, the equilibrium point $x_0$ for the ambient system $\tilde \Sigma$, then its restriction $u |M$ to $M$ also stabilizes, in the same sense, the equilibrium point for the restricted system $\Sigma | M$. 
\end{lemma}
\begin{theorem}\label{theorem:linear:point:stab}
If a linear feedback controller $u : \mathbb R^n \rightarrow \mathbb R^\dimu$ exponentially stabilizes the equilibrium point $x_0$ for the linearization $\tilde \Sigma^\ell_0$ of the ambient system $\tilde \Sigma$, then it also exponentially stabilizes the equilibrium point $x_0$ for $\Sigma |M$. 

\begin{proof}{\rm
Apply the Lyapunov linearization method and Lemma \ref{lemma:exp:stab}. 
}\end{proof}
\end{theorem}

Let us illustrate the above theorem with the free rigid body system given in \eqref{rigid:tilde:eq}. Choose any $R_0 \in \SO$. Then, $(R_0, 0) \in \SO \times \mathbb R^3$ is an equilibrium point of \eqref{rigid:tilde:eq} with $u = 0$. 
\begin{theorem}\label{theorem:rigid:linearization:point}
The linearization of  \eqref{rigid:tilde:eq} at $(R,\Omega) = (R_0, 0) \in \SO \times \mathbb R^3$ with $u_0 = 0$ is given by
\begin{subequations}\label{linear:rigid}
\begin{align}
\Delta \dot R &= R_0 \hat \Omega - 2\coV   R_0 \Sym (R_0^T\Delta R),\label{linear:rigid:a}\\
\dot \Omega &= u,\label{linear:rigid:b}
\end{align}
\end{subequations}
where 
\[
\Delta R = R -R_0.
\]
\end{theorem}
\begin{proof}
Equation \eqref{linear:rigid:a} can be easily derived by using the definition of the derivative as follows. Let $c(s) = R_0 + s (R-R_0) = R_0 + s\Delta R$ and $d(s) = s \Omega$, where $s\in \mathbb R$. Then
\begin{align*}
&\left . \frac{d}{ds} \right |_{s=0}  (c(s)\widehat {d(s)} - \coV c(s)(c(s)^Tc(s) - I))\\
&\qquad = R_0\hat\Omega - \coV  R_0 (\Delta R^T R_0 + R_0^T\Delta R)\\
&\qquad = R_0\hat\Omega - 2\coV R_0\Sym(R_0^T\Delta R),
\end{align*}
which is equal to the expression on the right side of \eqref{linear:rigid:a}.
\end{proof}

Let us change coordinates from $\Delta R$ to a new matrix variable $Z$ as follows:
\begin{equation}\label{def:ZR}
Z = R_0^T\Delta R.
\end{equation}
Let 
\begin{equation}\label{def:ZsZk}
Z_s = \Sym(Z), \quad Z_k = \Skew (Z)
\end{equation}
such that 
\begin{equation}\label{Z:Zs:Zk}
Z = Z_s + Z_k.
\end{equation}
  Multiplying \eqref{linear:rigid:a} by $R_0^T$ and taking the symmetric and skew symmetric parts, respectively,   transforms   \eqref{linear:rigid}  to 
\[
\dot Z_s = -2\coV Z_s,\quad
\dot Z_k = \hat \Omega, \quad \dot \Omega  = u,
\]
which can be also written as
\begin{equation}\label{rigid:Z:sys}
\dot Z_s = -2\coV Z_s,\quad
\dot Z_k^\vee =  \Omega, \quad \dot \Omega  = u.
\end{equation}

\begin{theorem}\label{theorem:linear:rigid:stab}
Take any two $3\times 3$ matrices $K_P$ and $K_D$ such that the following $6\times 6$ matrix 
\begin{equation}\label{point:stab:Hurwitz}
\begin{bmatrix}
0 & I \\
-K_P & -K_D
\end{bmatrix}
\end{equation}
is Hurwitz. Then, the linear PD controller
\begin{equation}\label{linear:control:rigid}
u(\Delta R,\Omega) = -K_P  \cdot \Skew (R_0^T\Delta R)^\vee - K_D\Omega
\end{equation}
 exponentially stabilizes the equilibrium point $(R_0,0)$ for the linearized system  \eqref{linear:rigid}. Moreover, it exponentially stabilizes the equilibrium point $(R_0,0)$ for the rigid body system \eqref{rigid:eq} on $\SO \times \mathbb R^3$.
\end{theorem}

\begin{proof}
Recall that the system in \eqref{linear:rigid} has been transformed to \eqref{rigid:Z:sys} by the state transformation \eqref{def:ZR} -- \eqref{Z:Zs:Zk}.
Take any $3\times 3$ matrices $K_P$ and $K_D$ such that the matrix in \eqref{point:stab:Hurwitz} is Hurwitz, and then apply  the following controller
\[
u = -K_P Z_k^\vee - K_D \Omega
\]
to the system \eqref{rigid:Z:sys}.
It is then easy to verify that the resultant closed-loop system is exponentially stable, which proves the first statement of the theorem. The second statement follows from the first statement and Theorem \ref{theorem:linear:point:stab}.  
\end{proof}

\begin{theorem}\label{theorem:PID:point:stab}
Take any three $3\times 3$ matrices $K_P$, $K_I$ and $K_D$ such that the following  polynomial in $\lambda$
\begin{equation}\label{PID:poly}
\det (\lambda^3 I + \lambda^2K_D + \lambda K_P + K_I) = 0
\end{equation}
is Hurwitz. Then, the linear PID  controller
\begin{align}
u(\Delta R,\Omega) &= -K_P  \cdot \Skew (R_0^T\Delta R)^\vee   - K_D\Omega \nonumber \\
&\quad\quad - K_I \int_0^t  \Skew (R_0^T\Delta R(\tau))^\vee  d\tau \label{PID:rigid}
\end{align}
 exponentially stabilizes the equilibrium point $(R_0,0)$ for the linearized system  \eqref{linear:rigid}. Moreover, it exponentially stabilizes the equilibrium point $(R_0,0)$ for the rigid body system \eqref{rigid:eq} on $\SO \times \mathbb R^3$.
\end{theorem}
\begin{proof}
Apply the controller  \eqref{PID:rigid} to  \eqref{rigid:Z:sys}  to get
\begin{align*}
\dot Z_s &= -2\coV Z_s,\\
\dot Z_k^\vee &=  \Omega, \\
\dot \Omega  &=  - K_PZ_k^\vee - K_D\Omega - K_I\int_0^tZ_k^\vee(\tau) d\tau,
\end{align*}
which is transformed, by differentiation, to
\begin{align*}
\dot Z_s &= -2\coV Z_s,\\
\dddot Z_k^\vee &+ K_D \ddot Z_k^\vee + K_P\dot Z_k^\vee + K_I  Z_k^\vee = 0.
\end{align*}
It is easy to prove that this linear system is exponentially stable by the Hurwitz condition on the polynomial in \eqref{PID:poly}.
This proves the first statement of the theorem. 

The second statement of the theorem can be proven by treating the PID controller \eqref{PID:rigid} as dynamic feedback to the nonlinear system \eqref{rigid:eq}  and then applying  the first statement of this theorem and Theorem \ref{theorem:linear:point:stab}. In this way, the Lyapunov linearization method is rigorously applied.  
\end{proof}

\begin{remark}{\rm
1. Since $R_0 \in \SO$, we have $\Skew (R_0^T\Delta R) = \Skew (R_0^T R - I) = \Skew (R_0^TR)$. Hence the controllers given in \eqref{linear:control:rigid} and \eqref{PID:rigid} can be written respectively as
\[
u = -K_P Z_k^\vee - K_D\Omega,
\]
and 
\begin{align*}
u &= -K_P  Z_k^\vee   - K_D\Omega - K_I \int_0^t Z_k^\vee(\tau)  d\tau
\end{align*}
with
\[
Z_k = \Skew(R_0^T R).
\]
These expressions avoid the computation $\Delta R = R-R_0$ that would be  an invalid operation on $\SO$.
It is interesting that this controller, though designed in Euclidean space, can be computed on $\SO \times \mathbb R^3$, which was not intended at the beginning of design.

2. Suppose that one chooses a finite number of points $R_0, \ldots, R_p$   from $\SO$ and plans a gain scheduling with the linearized systems at these points with $\Omega_0 = 0$.  If these points are not covered by one local chart on $\SO$, then one would need to change coordinates and re-do linearization on each change of coordinates. However,  in our scheme, only one form of  linearization in one single global Cartesian coordinate system, which is \eqref{linear:rigid}, is needed for all these points.
}\end{remark}

We now carry out a simulation to demonstrate a good performance of the stabilizing controller \eqref{linear:control:rigid} for the rigid body system \eqref{rigid:eq}, or equivalently \eqref{rigid:tilde:eq} with $\coV=1$, the latter of which is known to be better for numerical integration than the former; refer to  \cite{ChJiPe16} more about this.  Hence, we will use \eqref{rigid:tilde:eq} with $\coV=1$ for numerical integration. However, one can freely use \eqref{rigid:eq} instead for numerical integration.
The following control parameter values are chosen: 
\[
K_P = 4I, \quad K_D = 2I,
\]
 so that  the eigenvalues of \eqref{point:stab:Hurwitz} are placed at $-1\pm i \sqrt 3$.
The target equilibrium point $(R_0,\Omega_0) \in \SO \times \mathbb R^3$ is given by
\[
R_0 = \textrm{diag} \{-1,-1,1\}, \quad \Omega_0 = (0,0,0).
\]
The initial condition is chosen as 
\[
R(0) = \exp \left (\frac{2\pi}{3}\hat e_2 \right ), \quad \Omega(0) = (0,1,1),
\]
where $R(0)$ is a rotation about  $e_2=(0,1,0)$ through $2\pi/3$ radians.  The initial orientation error is computed as $\|R(0) - R_0\| = 2\sqrt 2$, which is the maximum possible error on $\SO$.  The magnitude errors of orientation and angular velocity  are plotted in Fig. \ref{figure.point.stab}, showing an excellent stabilizing performance of  the linear controller \eqref{linear:control:rigid} for the nonlinear system \eqref{rigid:eq}. 
\begin{figure}[bt]
\begin{center}
\includegraphics[scale = 0.37]{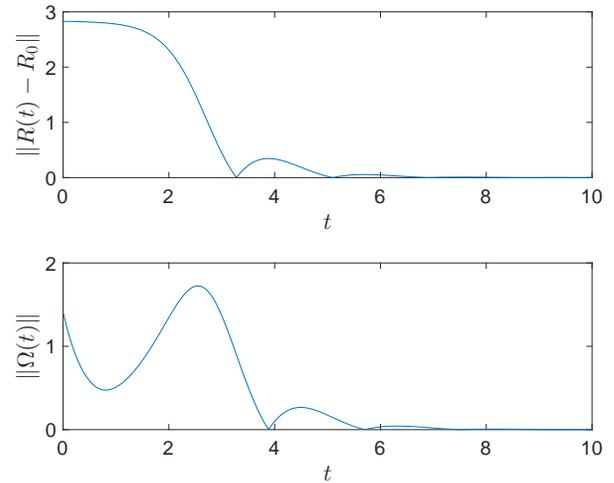}
\end{center}
\caption{\label{figure.point.stab}  The simulation result for  the stabilization of the equilibrium point $(R, \Omega ) = (R_0, 0)$ by the linear controller \eqref{linear:control:rigid} for the nonlinear rigid body system  \eqref{rigid:eq}.}
\end{figure}

\subsection{Tracking via Linearization in Ambient Euclidean Space}
Consider again the system $\tilde \Sigma$ given in \eqref{Sigma:tilde} and its restriction $\Sigma | M$ to $M$ given in \eqref{Sigma:M}, where  $\Sigma | M =\tilde  \Sigma | M$ is understood. Choose a reference trajectory $x_0: [0,\infty) \rightarrow M$ for $\Sigma |M$ on $M$ driven by a control signal $u_0: [0,\infty) \rightarrow \mathbb R^\dimu$, so that 
\[
\dot x_0(t) = \tilde X(x_0(t), u_0(t)) \quad \forall t\geq 0.
\]
We can then linearize the ambient system $\tilde \Sigma$ along the trajectory $(x_0(t), u_0(t))$  in $\mathbb R^n$ as follows:
\begin{equation}\label{Sigma:tilde:ell:t}
\tilde \Sigma^\ell_{\rm t}: \quad  \Delta \dot x = A(t) \Delta x+  B(t)\Delta u,
\end{equation}
where 
\[
A(t) = \frac{\partial \tilde X}{\partial x}(x_0(t),u_0(t)), \quad B(t) = \frac{\partial \tilde X}{\partial u}(x_0(t),u_0(t)) 
\]
and
\[
\Delta x = x - x_0 \in \mathbb R^n, \quad \Delta u = u - u_0 \in \mathbb R^\dimu.
\]
The following lemma is trivial but useful:
\begin{lemma}\label{lemma:exp:tracking}
If $u = u(x,t)$ is an exponentially tracking controller for the ambient system $\tilde \Sigma$ for the reference trajectory $x_0(t)$, then it is also an exponentially tracking controller for the system $ \Sigma | M$ on $M$ for the same reference trajectory.
\end{lemma}
\begin{theorem}\label{theorem:Khalil}
Suppose that a linear feedback controller   $\Delta u = - K(t)\Delta x$ exponentially stabilizes the origin for the linearized system $\tilde \Sigma_{\rm t}^\ell$.  Let $B_r = \{ z \in \mathbb R^n \mid \| z\| <r\}$ for some $r>0$ and $f:   B_r \times [0,\infty)  \rightarrow \mathbb R$ be a function defined by
\[
f(z,t) = \tilde X( x_0(t) + z, u_0(t) - K(t) z) - \tilde X(x_0(t),u_0(t)).
\]
If the derivative $\frac{\partial f}{\partial z}(z,t)$ is bounded and Lipschitz on $B_r$ uniformly in $t$, then the controller
\[
u(x,t) = u_0(t)  - K(t) (x - x_0(t))
\]
enables the system $\Sigma | M$ on $M$  to track the reference trajectory $x_0(t)$ exponentially. 

\end{theorem}

\begin{proof} 
Apply Theorem 4.13 in \cite{Kh02} and Lemma \ref{lemma:exp:tracking} above.  
\end{proof}

We illustrate the above theorem with the free rigid body system \eqref{rigid:tilde:eq}. Take a reference trajectory $(R_0(t), \Omega_0 (t)) \in \SO \times \mathbb R^3$ and the corresponding control signal $u_0(t)$ such that 
\begin{equation}\label{ref:traj:rigid}
\dot R_0(t) = R_0(t)\hat \Omega_0 (t), \quad \dot \Omega_0 (t) = u_0(t), \,\, \forall t \geq 0,
\end{equation}
which can be also understood as equations that define $\Omega_0(t)$ and $u_0(t)$ in terms of $R_0(t)$ and its time derivatives. Assume that  $(R_0(t), \Omega_0(t))$ and $u_0(t)$ are bounded  over the time interval $[0,\infty )$.
\begin{theorem}
The linearization of \eqref{rigid:tilde:eq} along the reference trajectory $(R_0(t),\Omega_0(t)) \in \SO \times \mathbb R^3$ and the reference control signal $u_0(t)$ is given by
\begin{subequations}\label{rigid:tilde:ell}
\begin{align}
\Delta \dot R &= \Delta R \hat\Omega_0+ R_0 \widehat{\Delta \Omega} - 2\coV  R_0 \Sym (R_0^T\Delta R),\\
\Delta\dot \Omega &= \Delta u,
\end{align}
\end{subequations}
where
\[
\Delta R = R - R_0(t) \in \mathbb R^{3\times 3}, \quad \Delta \Omega = \Omega - \Omega_0(t) \in \mathbb R^3,
\]
and
\[
 \Delta u = u - u_0(t) \in \mathbb R^3.
 \]

\end{theorem}
 \begin{proof}
This theorem can be proven with the same technique as that used for the proof of Theorem \ref{theorem:rigid:linearization:point}.
 \end{proof}

We now introduce a new matrix variable $Z$ replacing $\Delta R$ as follows:
\begin{equation}\label{ZR0R}
Z = R_0(t)^T \Delta R.
\end{equation}
 Let
\begin{equation}\label{def:Zs}
Z_s = \Sym (Z), \quad Z_k=\Skew (Z)
\end{equation}
such that 
\begin{equation}\label{def:Zk}
Z = Z_s + Z_k.
\end{equation}

\begin{lemma}
The system \eqref{rigid:tilde:ell} is transformed to
\begin{subequations}\label{Zs:Zk:Om:tv}
\begin{align}
\dot Z_s &= [Z_s, \hat \Omega_0] - 2\coV Z_s,\label{Zs:Zk:Om:tv:a}\\
\dot Z_k^\vee &= Z_k^\vee \times \Omega_0 + \Delta \Omega,  \label{Zs:Zk:Om:tv:b}\\
\Delta \dot \Omega &= \Delta u \label{Zs:Zk:Om:tv:c}
\end{align}
\end{subequations}
via the state transformation given in \eqref{ZR0R} -- \eqref{def:Zk}.
\end{lemma}
\begin{proof}
Differentiate both sides of \eqref{ZR0R}, and use \eqref{rigid:tilde:ell},   \eqref{ZR0R} -- \eqref{def:Zk}, and  \eqref{ref:traj:rigid}  to obtain
\begin{align*}
\dot Z &= \dot R_0^T \Delta R + R_0^T\Delta \dot R \\
&= -\hat\Omega_0R_0^T\Delta R + R_0^T \Delta R \hat \Omega_0 + \widehat {\Delta \Omega} - 2\coV \Sym(R_0^T\Delta R)\\
&= [Z, \hat \Omega_0] + \widehat {\Delta \Omega} - 2\coV \Sym(Z)\\
&= [Z_s,  \hat \Omega_0]  + [Z_k,  \hat \Omega_0] + \widehat {\Delta \Omega} - 2\coV Z_s.
\end{align*}
Taking the symmetric and skew-symmetric parts, we get
\begin{align*}
\dot Z_s =  [Z_s,  \hat \Omega_0]  - 2\coV  Z_s,\quad \dot Z_k =   [Z_k,  \hat \Omega_0] + \widehat {\Delta \Omega},
\end{align*}
where the second equation can be also written as \eqref{Zs:Zk:Om:tv:b}.
This completes the proof. 
\end{proof}

\begin{theorem}\label{theorem:delta:1}
For any two matrices $K_P, K_D \in \mathbb R^{3\times 3}$ such that the matrix in \eqref{point:stab:Hurwitz} becomes Hurwitz,  the linear controller
\begin{align}
\Delta u &= -K_P \cdot Z_k^\vee - K_D (Z_k^\vee \times \Omega_0 + \Delta \Omega )  \nonumber \\
&\quad\quad - (Z_k^\vee \times \Omega_0 + \Delta \Omega )\times \Omega_0 - Z_k^\vee \times u_0 \label{Delta:u:rigid}
\end{align}
exponentially stabilizes the origin for the system \eqref{Zs:Zk:Om:tv}. 
\end{theorem}
\begin{proof}
Let us first show exponential stability of the subsystem \eqref{Zs:Zk:Om:tv:a} that is decoupled from the rest of the system. Let $V(Z_s) = \frac{1}{2}\|Z_s\|^2$. Along the trajectory of  \eqref{Zs:Zk:Om:tv}, $\frac{d}{dt}V = \langle Z_s, [Z_s, \hat\Omega_0]  - 2\coV Z_s\rangle = -2\coV \|Z_s\|^2 = -4\coV V$, where it is easy to show $\langle Z_s, [Z_s, \hat\Omega_0]  \rangle = 0$. Hence, $V(t) \leq e^{-4\coV t}V(0)$ for all $t\geq 0$, or 
\begin{equation}\label{Zs:exp}
\|Z_s(t)\| \leq e^{-2\coV t} \|Z_s(0)\|
\end{equation} 
for all $t\geq 0$ and $Z_s(0) \in \Sym(\mathbb R^{3\times 3})$, which proves exponential stability of $Z_s = 0$ for  \eqref{Zs:Zk:Om:tv:a}.

Differentiating \eqref{Zs:Zk:Om:tv:b} and substituting \eqref{Zs:Zk:Om:tv:c}  transforms the subsystem \eqref{Zs:Zk:Om:tv:b} and \eqref{Zs:Zk:Om:tv:c} to the following second-order system:
\[
\ddot Z_k^\vee = \dot Z_k^\vee \times \Omega_0 +Z _k^\vee \times u_0 + \Delta u
\]
since $\dot \Omega (t) = u_0(t)$. This second-order system
 is exponentially stabilized by the controller
\begin{equation}\label{Delta:u:another:form}
\Delta u = -K_P \cdot Z_k^\vee - K_D \dot Z_k^\vee -  \dot Z_k^\vee \times \Omega_0 - Z_k^\vee \times u_0,
\end{equation}
where the matrices $K_P, K_D \in \mathbb R^{3\times 3}$ are any matrices such that the matrix in \eqref{point:stab:Hurwitz} becomes Hurwitz. So, there are positive constants $C_1$ and $C_2$ such that 
\[
\|Z_k^\vee (t)\| + \|\dot Z_k^\vee (t)\| \leq C_1e^{-C_2t} (\|Z_k^\vee (0)\| + \|\dot Z_k^\vee (0)\| )
\]
for all $t\geq 0$ and $(Z_k^\vee (0), \dot Z_k^\vee (0)) \in \mathbb R^{3} \times  \mathbb R^{3}$.
Since $\Omega_0(t)$ is bounded by assumption, there is a constant $M>0$ such that $ \|\Omega_0(t)\| \leq M$ for all $t\geq0$. By  \eqref{Zs:Zk:Om:tv:b} and the triangle inequality,
\[
\|\dot Z_k^\vee (t) \| \leq   M\|Z_k^\vee (t)\| + \| \Delta \Omega (t)\|
\]
and
\[
 \|\Delta \Omega (t)\| \leq \| \dot Z_k^\vee (t)\| + M \|Z_k^\vee (t)\| 
\]
for all $t\geq 0$. It is then easy to show that
\begin{align}
\|Z_k^\vee (t)\| &+ \|\Delta \Omega (t)\| \nonumber \\
&\leq C_3e^{-C_2t}(\|Z_k^\vee (0)\| + \|\ \Delta \Omega (0)\|)  \label{Zk:Om:exp}
\end{align}
for all $t\geq 0$ and $(Z_k^\vee (0), \Delta \Omega (0))\in \mathbb R^{3} \times  \mathbb R^{3}$, where $C_3 = C_1(1+M)^2$. Notice that the controller given in \eqref{Delta:u:another:form} is the same as that in \eqref{Delta:u:rigid}. Hence,  it follows  
from \eqref{Zs:exp} and \eqref{Zk:Om:exp} that the controller  \eqref{Delta:u:rigid} exponentially stabilizes the origin for the system \eqref{Zs:Zk:Om:tv}. 
\end{proof}

\begin{theorem}\label{theorem:delta:2}
For any three matrices $K_P, K_D, K_I \in \mathbb R^{3\times 3}$ such that the polynomial in \eqref{PID:poly} becomes Hurwitz,  the linear controller
\begin{align}
\Delta u &= -K_P \cdot Z_k^\vee - K_D (Z_k^\vee \times \Omega_0 + \Delta \Omega )  \nonumber \\
&\quad\quad - K_I\int_0^t   Z_k^\vee (\tau) d\tau  - (Z_k^\vee \times \Omega_0 + \Delta \Omega )\times \Omega_0 \nonumber \\
&\quad\quad- Z_k^\vee \times u_0 \label{Delta:u:rigid:PID}
\end{align}
exponentially stabilizes the origin for the linear time-varying system \eqref{Zs:Zk:Om:tv}. 
\end{theorem}


\begin{theorem}\label{theorem:delta:3}
For any positive number $k_P$ and any positive definite symmetric matrix $K_D \in \mathbb R^{3\times 3}$, the PD controller
\begin{equation}\label{delta:u:kpd}
\Delta u = -k_PZ_k^\vee - K_D \Delta \Omega
\end{equation}
exponentially stabilizes the origin for the  system \eqref{Zs:Zk:Om:tv}. 
\end{theorem}
\begin{proof}
It is straightforward to prove this theorem with the Lyapunov function
\[
V = \frac{1}{2}\|Z_s\|^2 + \frac{k_P}{2} \|Z_k^\vee\|^2 + \frac{1}{2}\|\Delta \Omega\|^2 + \epsilon \langle Z_k^\vee,\Delta  \Omega  \rangle
\]
and the Lyapunov arguments used in \cite{LeLeMc13}. 
\end{proof}

The following theorem is a variant of Theorem \ref{theorem:delta:3}.
\begin{theorem}\label{theorem:delta:4}
For any two positive numbers $k_P$ and $\epsilon$ and any positive definite symmetric matrix $K_D \in \mathbb R^{3\times 3}$ such that
\[
0 < \epsilon < \min \left \{ \sqrt{k_P}, \frac{4k_P\lambda_{\min}(K_D)}{4k_P + (\lambda_{\max}(K_D))^2}\right \},
\]
 the controller
\begin{equation}\label{delta:u:kpd:Omega0}
\Delta u = -k_P Z_k^\vee - K_D \Delta \Omega - \epsilon (Z_k^\vee \times \Omega_0)
\end{equation}
exponentially stabilizes the origin for the system \eqref{Zs:Zk:Om:tv}. 

\end{theorem}

The following theorem puts together the four preceding theorems to provide tracking controllers for the rigid body system \eqref{rigid:eq}.
\begin{theorem}
Consider the following controller
\begin{equation}\label{THE:tracking:control:rigid}
u = u_0 + \Delta u,
\end{equation}
where $\Delta u$ is any of \eqref{Delta:u:rigid},  \eqref{Delta:u:rigid:PID}, \eqref{delta:u:kpd} and \eqref{delta:u:kpd:Omega0} with 
\[
Z_k = \Skew (R_0^T\Delta R)^\vee = \Skew(R_0^TR)^\vee.
\]
 Then, it enables the free rigid body system given in  \eqref{rigid:eq} to track the reference trajectory $(R_0(t), \Omega_0(t))$ exponentially.
\end{theorem}

We carry out a simulation to show an excellent tracking performance of the controller  \eqref{THE:tracking:control:rigid} and \eqref{delta:u:kpd:Omega0}  for the rigid body system \eqref{rigid:eq} or  \eqref{rigid:tilde:eq} with $\coV =1$.
The control parameters are chosen as 
\[
k_P=4, \quad K_D = 2I, \quad \epsilon = 1.
\]
The reference trajectory $(R_0(t), \Omega_0(t)) \in \SO \times \mathbb R^3$ with the reference control signal $u_0(t) \in \mathbb R^3$ are chosen as 
\begin{align*}
R_0(t)  &\\
&\!\!\! \!\!\!\!\!\! \!\!\!\!{\footnotesize = \begin{bmatrix}
\cos^2 t & - \sin t & \cos t\sin t \\
\sin^2 t + \cos^2 t\sin t & \cos^2 t & \cos t \sin^2 t - \cos t \sin t \\
\cos t \sin^2t - \cos t \sin t & \cos t \sin t & \cos^2t + \sin^3t 
\end{bmatrix}}\\
\Omega_0 (t)&= \begin{bmatrix}
\cos^2 t - \sin t \\ 1-\sin t \\ \cos t (1+  \sin t)
\end{bmatrix}, \\
 u_0(t) &= 
 \begin{bmatrix}
-2\cos t \sin t - \cos t \\ -\cos t \\ -\sin t (1+\sin t) + \cos^2 t
\end{bmatrix},
\end{align*}
which satisfy \eqref{ref:traj:rigid}. The initial condition is given by
\[
R(0) = \exp (0.99\pi \hat e_2), \quad \Omega (0) =  (1,1,1),
\]
where $R(0)$ is a rotation around $e_2 = (0,1,0)$ through $0.99\pi$ radians. The initial orientation tracking error is almost $2\sqrt 2$. The tracking errors  are plotted in Fig. \ref{figure.tracking:rigid}, which shows the excellent tracking performance of the linear controller for the nonlinear system \eqref{rigid:eq}.

\begin{figure}[tb]
\begin{center}
\includegraphics[scale = 0.37]{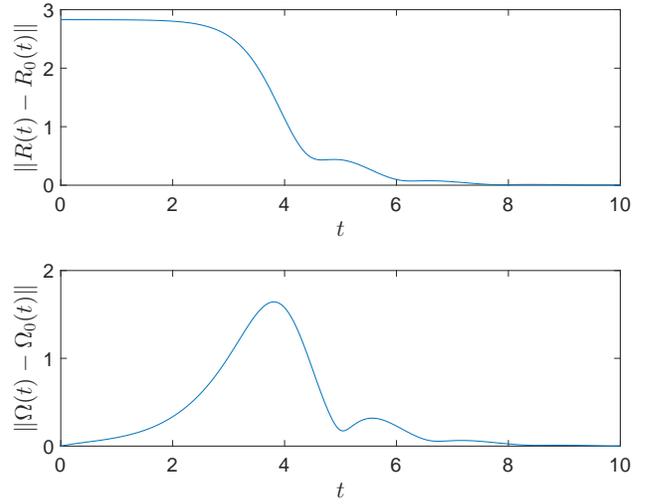}
\end{center}
\caption{\label{figure.tracking:rigid}  The simulation result of tracking the reference $(R_0(t), \Omega_0(t))$ by the linear controller  \eqref{THE:tracking:control:rigid} with \eqref{delta:u:kpd:Omega0}  for the rigid body system \eqref{rigid:eq}.}
\end{figure}

\section{Future Work}
We plan to extend our program to the design of observers and filters. 
%


\end{document}